\documentclass{amsart}
\usepackage{amsmath,amsfonts,amssymb,latexsym}
\usepackage{graphicx}
\usepackage{fancyhdr}
\usepackage{multirow}

\usepackage{a4wide}

\usepackage{cite}

\usepackage{listings}

\usepackage{hyperref}

\hypersetup{%
pdftitle={},
pdfsubject={Mathematics},
pdfauthor={},
pdfkeywords={},
hyperindex=true,plainpages=false}

\newcommand{\Q}{{\mathbb Q}}
\newcommand{\C}{{\mathbb C}}
\newcommand{\N}{{\mathbb N}}
\newcommand{\Z}{{\mathbb Z}}
\newcommand{\p}{{\mathfrak p}}
\renewcommand{\P}{{\mathbb P}}
\newcommand{\ord}{\mathrm{ord}}
\renewcommand{\Re}{\mathrm{Re}}

\DeclareMathOperator{\Norm}{N}

\theoremstyle{plain}
\numberwithin{equation}{section}
\newtheorem{thm}{Theorem}[section]
\newtheorem{theorem}[thm]{Theorem}

\newtheorem{lemma}[thm]{Lemma}

\newtheorem{proposition}[thm]{Proposition}
\newtheorem{conjecture}[thm]{Conjecture}

\newtheorem{cor}[thm]{Corollary}

\theoremstyle{remark}
\newtheorem{remark}{Remark}

\begin{document}

\title[On multiplicatively independent bases]
{On multiplicatively independent bases in cyclotomic number fields}

\author[M. G. Madritsch]{Manfred G. Madritsch}
\address{Manfred Madritsch\newline
\noindent 1. Universit\'e de Lorraine, Institut Elie Cartan de Lorraine, UMR 7502, Vandoeuvre-l\`es-Nancy, F-54506, France;\newline
\noindent 2. CNRS, Institut Elie Cartan de Lorraine, UMR 7502, Vandoeuvre-l\`es-Nancy, F-54506, France}
\email{manfred.madritsch@univ-lorraine.fr}

\author{Volker Ziegler}
\address{Volker Ziegler\newline
\noindent Johann Radon Institute for Computational and Applied Mathematics (RICAM)\newline
\noindent Austrian Academy of Sciences\newline
\noindent Altenbergerstr. 69\newline
\noindent A-4040 Linz, Austria}
\email{volker.ziegler\char'100ricam.oeaw.ac.at}
\thanks{The second author was supported by the Austrian Science Fund (FWF) under the project P~24801-N26.}

\begin{abstract}
  Recently the authors \cite{Madritsch:2014} showed that the algebraic
  integers of the form $-m+\zeta_k$ are bases of a canonical number
  system of $\Z[\zeta_k]$ provided $m\geq \phi(k)+1$, where $\zeta_k$
  denotes a $k$-th primitive root of unity and $\phi$ is Euler's
  totient function. In this paper we are interested in the questions
  whether two bases $-m+\zeta_k$ and $-n+\zeta_k$ are multiplicatively
  independent. We show the multiplicative independence in case that
  $0<|m-n|<10^6$ and $|m|,|n|> 1$.
\end{abstract}


\maketitle

\section{Introduction}

Let $q\geq2$ be an integer. Then every positive integer $n\in\N$ has a
unique representation of the form
\[n=\sum_{k=0}^\ell d_kq^k \quad\text{with}\quad
d_k\in\{0,\ldots,q-1\}.\] We call $q$ the base and
$\mathcal D=\{0,\ldots,q-1\}$ the set of digits. This representation can be
extended to all integers (positive and negative ones) by taking $q\leq
-2$ and $\{0,\ldots,\lvert q\rvert-1\}$. In general we call such a
pair $(q,\mathcal D)$ a numeration system in $\Z$.

Knuth \cite{knuth1981:art_computer_programming} was one of the first
extending the concept of numeration systems to the Gaussian integers by showing that the bases
$-1+i$ and $-1-i$ together with the set $\{0,1\}$ form a number
system. Later K\'atai and Szab\'o
\cite{katai_szabo1975:canonical_number_systems} showed that all bases
are of the form $-m\pm i$ with set of digits
$\mathcal D=\{0,1,\ldots,m^2\}$. This was further extended independently to all
quadratic number fields by Gilbert~\cite{gilbert81:_radix} and K\'atai
and Kov\'acs~\cite{katai_kovacs1980:kanonische_zahlensysteme_in,
  katai_kovacs1981:canonical_number_systems}.

A different point of view is the following. Let $f$ be a polynomial
with coefficients in $\Z$. Then we call the pair $(X,\{0,\ldots,\lvert
f(0)\rvert-1\})$ a canonical numeration system in $\mathcal R=\Z[X]/(f)$ if every element
$g\in\Z[X]/(f)$ has a unique representation of the form
\[g=\sum_{k=0}^\ell d_kX^k
\quad\text{with}\quad
d_k\in\{0,\ldots,\lvert f(0)\rvert-1\}.\]
By taking $f=X^2+2mX+m^2+1$ and $f=X+q$ we get back the numeration systems
$(-m\pm i,\{0,1,\dots,m^2\})$ in $\Z[i]$ and $(-q,\{0,1,\dots,|q|-1\})$ in $\Z$ respepectively.

Let $f=X^d+p_{d-1}X^{d-1}+\cdots+p_1X+p_0$ be an irreducible polynomial over $\Z$ with root $\beta$, 
then Kov\'acs
\cite{kovacs1981:canonical_number_systems} could prove that
$\beta$ is a base of a canonical numeration system if $1\leq p_{d-1}\leq\cdots\leq p_1\leq p_0$ and
$p_0\geq2$. Peth{\H o} \cite{pethoe1991:polynomial_transformation_and}
replaced the irreducibility condition by
supposing that none of the roots of $p$ is a root of unity. Moreover, 
Kov\'acs \cite{kovacs1981:canonical_number_systems} showed that for
any given order $\Z[\alpha]$ in a number field there exists an element
$\beta\in\Z[\alpha]$ such that $(\beta,\{0,1,\ldots,\Norm(\beta)-1)$ is a
numeration system in $\Z[\alpha]$, where $\Norm(\beta)$ denotes the absolute norm of $\beta$. Kov\'acs and Peth{\H o}
\cite{kovacs_petho1991:number_systems_in} gave effectively computable,
necessary and sufficient conditions for the element $\beta$ in order
to be the base of a numeration system. The underlying algorithm was
improved by Akiyama and Peth{\H o}
\cite{akiyama_petho2002:canonical_number_systems} for the special
class of polynomials satisfying
\[\sum_{i=1}^d\lvert p_i\rvert<p_0.\]

\section{Definitions and Statement of results}
We want to take a further look at numeration systems in the Gaussian
integers. Since $\pm i$ are primitive fourth roots of unity, one may
write $-m\pm i$ as $-m+\zeta_4$. Similarly one may see numeration systems in the integers as having bases $-m+\zeta_2=-(m+1)$. With the above mentioned
characterization for quadratic number fields by Gilbert, K\'atai and
Kov\'acs \cite{gilbert81:_radix,
  katai_kovacs1981:canonical_number_systems,
  katai_kovacs1980:kanonische_zahlensysteme_in} it is easy to show,
that algebraic integers of the form $-m+\zeta_3$ and $-m+\zeta_6$ are bases for numeration systems in $\Z[\zeta_3]$ and $\Z[\zeta_6]$, respectively.

In a recent paper \cite{Madritsch:2014} the authors extended this to
arbitrary cyclotomic number fields.
\begin{theorem}[{\cite[Theorem 1.1]{Madritsch:2014}}]
  Let $k>2$ and $m$ be positive integers and $\zeta_k$ be a primitive
  $k$-th root of unity. If $m>\varphi(k)$, then $-m+\zeta_k$ is a base
  of a numeration system in $\Z[\zeta_k]$.
\end{theorem}

Their second result in that paper \cite{Madritsch:2014} concerns the
multiplicative independence of these bases.  We call two algebraic
integers $\alpha$ and $\beta$ multiplicatively independent if the
equation $\alpha^p=\beta^q$ has only the trival solution $p=q=0$ over
the integers.

\begin{theorem}[{\cite[Theorem 1.2]{Madritsch:2014}}]\label{th:old_mult}
  Let $k>2$ be a positive integer and $\zeta_k$ be a primitive $k$-th
  root of unity. Then the algebraic integers $-m+\zeta_k$ and
  $-n+\zeta_k$ are multiplicatively independent provided $m>n>C(k)$,
  where $C(k)$ is an effectively computable constant depending on $k$.

  Moreover, if $k$ is a power of $2,3,5,6,7,11,13,17,19$ or $23$, then
  $-m+\zeta_k$ and $-n+\zeta_k$ are multiplicatively independent as
  long as $m>n>1$.
\end{theorem}

In particular, we conjecture that the second part of the last theorem
holds true for all $k$.

The statement of Theorem \ref{th:old_mult} could be seen as fixing some
$k$ and let $m$ and $n$ vary over the positive integers. Therefore the
aim of the present paper is to prove the other direction. In
particular, we fix the distance of $m$ and $n$ by $a$ and let $k$ vary
over the positive integers.

However, before we state our main theorem we need to introduce some
further notation. We define the order of $a$ modulo $k$ with
$\gcd(a,k)=1$ as the smallest positive integer $n$ such that
$a^n\equiv 1 \mod k$ and write $n=\ord_k(a)$. Note that in the
extremal case $k=1$ we have $\ord_1(a)=1$ for all positive integers
$a$. Furthermore, for every prime $p$ we denote by $\nu_p(a)$ the
$p$-adic valuation of an integer $a$. Then our main result is the
following:

\begin{theorem}\label{th:mult_ind}
  Let $a>0$ be a given integer and assume that $-m+\zeta_k$ and
  $-(m+a)+\zeta_k$ with $k>2$ and $m$ an integer are multiplicatively
  dependent and none is a root of unity. Furthermore let
  \[S=\left\{p\in\P\: :\: p|\Phi_k(m)\Phi_k(m+a)\right\},\]
  where $\Phi_k$ is the $k$-th cyclotomic polynomial.
  Then, either $m=-1$, $a=2$ and $k=4$ or 
\begin{itemize}
 \item $S\subset S_a$, where $S_a$ is the set of prime divisors of $a$ and $S\neq \emptyset$.
 \item If $k\not\equiv 2 \mod 4$, then $k\leq \min_{p\in S}
   \{p^\alpha\: :\:p^\alpha \| a\}$, $p\nmid k$ and $\ord_k(p)\leq
   \nu_p(a)$ for all $p\in S$.
 \item Let $f_p=\ord_{k}(p)$ for all $p\in S$. Then
   \[\Phi_k(x) \left| \prod_{p\in S} p^{\lfloor \nu_p(a)/f_p \rfloor f_p}\right. ,\]
   where either $x=m$ or $x=m+a$. In particular,
   $f_p|\nu_p(\Phi_k(x))$ for all $p\in S$.
\end{itemize}
\end{theorem}

\begin{remark}\label{rem:k_not_2}
  We note that without loss of generality we may assume that
  $k\not\equiv 2 \mod 4$. Indeed if $k\equiv 2 \mod 4$ we have
  $-m+\zeta_k=-m-\zeta_{k/2}=(-1)(-(-m)+\zeta_{k/2})$, where
  $\zeta_{k/2}$ is a suitable $k/2$-th primitive root of
  unity. Therefore $-m+\zeta_k$ and $-(m+a)+\zeta_k$ are
  multiplicatively independent if and only if $-(-m)+\zeta_{k/2}$ and
  $-(-(m+a))+\zeta_{k/2}$ are.
\end{remark}

\begin{remark}
We want to emphasize that in case that $m\leq0$ the algebraic integer $-m+\zeta_k$ cannot be a base for $\Z[\zeta_k]$.
Therefore in view of canonical number systems only the case $m\geq1$ is of interest. However for the sake of generality
we do not restrict ourselves to the case $m\geq1$ in Theorem~\ref{th:mult_ind}.
\end{remark}

An obvious consequence of Theorem \ref{th:mult_ind} is the following Corollary

\begin{cor}\label{Cor:finite}
  For given $a>0$ there exist at most finitely many pairs $(m,k)\in
  \Z^+\times \Z_{\geq 3}$ such that $-m+\zeta_k$ and $-(m+a)+\zeta_k$ are
  multiplicatively dependent.  Moreover the pairs $(m,k)$ are
  effectively computable.
\end{cor}

Let us remark that applying a result due to Bombieri et. al. \cite[Theorem 1]{Bombieri:1999} (see also \cite[Theorem~3.22]{Zannier:LNDA}) yields 
a result similar to Corollary \ref{Cor:finite}.

\begin{theorem}[Bombieri et. al. \cite{Bombieri:1999}]\label{th:mult-general}
  Let $\mathcal C$ be a curve defined over $\bar\Q$ and let
  $\phi_1,\dots,\phi_n \in \bar\Q(\mathcal C)$ be multiplicatively
  independent modulo $\bar\Q^*$. Then the set of points $P\in \mathcal
  C (\bar \Q)$ such that $\phi_1(P),\dots,\phi_n(P)$ are
  multiplicatively dependent is of bounded height.
\end{theorem}

\begin{remark} Let $\mathcal C\subset \mathbb A^2$ be the line $\bar\Q[X,Y]/(Y)$ and
$\phi_1(X,Y)=X$ and $\phi_2(X,Y)=X+a$, then according to Theorem
\ref{th:mult-general} all $x=-m+\zeta_k$ such that $x=-m+\zeta_k$ and
$x-a=-(m+a)+\zeta_k$ are multiplicatively dependent are of bounded
height, \textit{i.e.} there exists a constant $C_a$ depending only on $a$ such
that $|m|<C_a$. Unfortunately Theorem \ref{th:mult-general} does not
provide any information on $k$. However, let us note that $C_a$ is an
effectively computable constant and by dully reproving
\cite[Theorem 3.22]{Zannier:LNDA} for this specific instance, but with
keeping track of error terms we obtain $C_a=a^{10}+1$ provided $a>10$,
which yields a much worse bound, than the bound obtained by applying
Theorem~\ref{th:mult_ind}.
\end{remark}

As mentioned above Theorem \ref{th:mult_ind} provides an effective algorithm to determine all pairs $(m,k)$ such that for given 
$a$ the algebraic integers $-m+\zeta_k$ and $-(m+a)+\zeta_k$ are multiplicatively dependent. Moreover this algorithm is also rather efficient 
as the next corollary indicates.

\begin{cor}\label{Cor:small_a}
  Let $1\leq a \leq 10^6$ and assume that $m\neq 0,-a$ and $(m,k)\neq
  (1,6),(-1,3),(-a+1,6),(-a-1,3)$. Then $-m+\zeta_k$ and
  $-m-a+\zeta_k$ are multiplicatively independent or $m=-1$, $a=2$ and $k=4$.
\end{cor}

This result together with the various results in \cite{Madritsch:2014}
concerning this topic encourages us to state the following
conjecture

\begin{conjecture}
  For every $k>2$ and $m>n>1$ the algebraic integers $-m+\zeta_k$ and
  $-n+\zeta_k$ are multiplicatively independent.
\end{conjecture}

The verification of Corollary \ref{Cor:small_a} is based on an
algorithm indicated by Theorem \ref{th:mult_ind}. However, the crucial
point is to generate a preferably short list of possible $k$'s. But
the list of possible $k$'s generated by Theorem \ref{th:mult_ind} is
rather long for large $a$. Using such a long list makes it unfeasible
to show the multiplicative independence of $-m+\zeta_k$ and
$-(m+a)+\zeta_k$ for large $a$. Therefore instead of using Theorem
\ref{th:mult_ind} directly we compute for each nonempty $S\subset S_a$
a list of possible $k$'s, which is in most cases very short or even
empty. How to compute such short lists and other computational issues
are discussed in Section~\ref{Sec:Algorithm}.


\section{Multiplicative independence}\label{Sec:Mult_Ind}

We may assume that neither $-m+\zeta_k$ nor $-(m+a)+\zeta_k$ are roots of unity
and, as indicated by Remark~\ref{rem:k_not_2}, that $k$ is odd or $4|k$.

Let us start with a simple fact about cyclotomic polynomials.

\begin{lemma}\label{lem:Cyclotomic_pol}
 Assume that $k\not\equiv 2 \mod 4$. Then $\Phi_k(x)=\pm 1$ if and only if 
 \begin{itemize}
  \item  $x=0$, or
  \item $x=1$ and $k\neq p^\ell$, with a prime $p$ and $\ell>0$, or
  \item $x=-1$ and $k\neq 2^\ell$, with $\ell>1$.
 \end{itemize} 
  Moreover we have $\Phi_k(1)=p$ if $k=p^\ell$ and $\Phi_{2^\ell}(-1)=2$, if $\ell>1$.
\end{lemma}

\begin{proof}
  The computation of $\Phi_k(1)$ can be done
  by M{\"o}bius inversion formula applied to
  $$k=\lim_{x\rightarrow 1} \frac{x^k-1}{x-1} =\prod_{d|k,    d\neq 1} \Phi_d(1).$$
  See also the proof of the corollary of Theorem 1 in Chapter IV of
  \cite{Lang:ANT}. In particular this yields $\Phi_k(1)=p$ if $k=p^\ell$ is a
  prime power and $\Phi_k(1)=1$ otherwise (see also
  \cite[Proposition 3.5.4]{Cohen:NTI}).
  
  Moreover, we know that if $k$
  is odd we have $\Phi_k(-x)=\Phi_{2k}(x)$ and
  $\Phi_{2^\ell}(x)=x^{2^{\ell-1}}+1$, with $\ell>0$. Hence all the
  values of $\Phi_k(\pm 1)$ are as claimed in the lemma provided
  $k\not\equiv 2 \mod 4$.
  
  The computation of $\Phi_k(0)=1$ is again easy. So we are left to
  prove that $|\Phi_k(x)|> 1$ for any integer $|x|\geq 2$.  But, in
  this case we have
  $$|\Phi_k(x)|=\prod_{\substack{1 \leq a < k\\\gcd(a,k)=1}} |x-\zeta_k^a|> \prod_{\substack{1 \leq a <k\\ \gcd(a,k)=1}} 1.$$
\end{proof}

In the next step we show that we may exclude the case that both $-m+\zeta_k$ and $-(m+a)+\zeta_k$ are units.

\begin{lemma}\label{lem_not:units}
 Assume that $-m+\zeta_k$ and $-(m+a)+\zeta_k$ are both units and multiplicatively dependent. Then one of them is a root of unity. In particular $S\neq 
\emptyset$ in Theorem~\ref{th:mult_ind}.
\end{lemma}

\begin{proof}
  Since the norm of $-m+\zeta_k$ is $\Phi_k(m)$ we deduce form Lemma
  \ref{lem:Cyclotomic_pol} that $-m+\zeta_k$ is a unit only if $m=0$
  or $m=\pm 1$ and $k$ is not a prime power. Therefore we conclude
  that $a=1$ or $a=2$. However, if $a=1$ and both $-m+\zeta_k$ and
  $-m-a+\zeta_k$ are units, then we always obtain that either $m=0$ or
  $m+a=0$, \textit{i.e.} either $-m+\zeta_k$ or $-m-a+\zeta_k$ is a
  root of unity.

Therefore we may assume that $k$ is not a prime power and $a=2$. We have to investigate the Diophantine equation
\begin{equation}\label{eq:a2_unit}
(-1+\zeta_k)^r= (1+\zeta_k)^s.
\end{equation}
In particular, we show that if neither $-1+\zeta_k$ nor $1+\zeta_k$ is a root of unity, then $r=s=0$ is the only solution to \eqref{eq:a2_unit}.
As already mentioned in Remark \ref{rem:k_not_2} we may assume that either $4|k$ or $k$ 
is odd. First, let us assume that $k$ is odd. In this case equation \eqref{eq:a2_unit} can be rewritten as
$$ (-1+\zeta_k)^{r+s}=(-1+\zeta_k^2)^{s}.$$
Since we assume that $k$ is odd $-1+\zeta_k^2$ and $-1+\zeta_k$ are conjugate and have therefore the same height $H$. By assumption $-1+\zeta_k$ is not a root 
of unity and we have $H>1$ and $H^{|r+s|}=H^{|s|}$. Therefore either $r=0$ or $r=-2s$ holds. The first case yields $(1+\zeta_k)^s=1$, \textit{i.e.} $1+\zeta_k$ is a 
root of unity or $s=0$. The second case yields
$$ (-1+\zeta_k)^s(-1+\zeta_k^2)^s=1$$
and by taking $s$-th roots we obtain 
\begin{equation}\label{eq:roots_special_case}
(-1+\zeta_k)(-1+\zeta_k^2)=\zeta_k^\ell
\end{equation}
for some integer $\ell$ such that $\zeta_k^\ell$ is an $s$-th root of unity. Let us embed the cyclotomic field $\Q(\zeta_k)$ into $\C$ 
by $\zeta_k\mapsto \exp(2\pi i/k)$. It is well known that $\exp(ix)=1+\theta$ with $|\theta|<|x|$ for real $x$.
Therefore we obtain $|-1+\exp(4\pi i/k)|<1/2$ provided $k\geq 26$. On the other hand obviously $|1+\exp(2\pi i/k)|<2$, hence the left hand side of 
equation~\eqref{eq:roots_special_case} is $<1$ in absolute values, \textit{i.e.} a contradiction. Therefore we may assume that $2<k\leq 25$ is odd and no prime power, 
\textit{i.e.} $k=15$ or $k=21$. A direct verification in these two cases shows that equation~\eqref{eq:roots_special_case} does not hold either.

Let us turn to the case that $4|k$ but $k$ is not a power of $2$. Let us write $k=4n$. In this case we have $-\zeta_k=\zeta_k^{1+2n}$. Since 
$\gcd(4n,2n+1)=1$ we deduce that $(-1+\zeta_k)$ and $(1+\zeta_k)$ are conjugate and by the height argument above we obtain that $r=\pm s$. Therefore either
$1-\zeta_k^2$ is a root of unity or $\frac{1+\zeta_k}{1-\zeta_k}$ is a root of unity. Since the equation $1=x+y$ with $|x|=|y|=1$ has only sixth roots of unity 
as solutions we are left to the case that $\frac{1+\zeta_k}{1-\zeta_k}$ is a root of unity. But $|1+\exp(2\pi i/k)|>1$ and $|1-\exp(2\pi i/k)|<1$ unless 
$2< k\leq 6$. But the only integer $k$ in this range divisible by $4$ is $4$, a prime power of $2$.
\end{proof}

According to Theorem \ref{th:mult_ind} we suppose that
$-m+\zeta_k$ and $-(m+a)+\zeta_k$ are multiplicatively dependent. Then
every prime ideal $\p$ dividing the principal ideal $(-m+\zeta_k)$
also divides the principal ideal $(-(m+a)+\zeta_k)$ and hence the principal ideal
$(a)$. Note that such a prime ideal $\p$ exists due to
Lemma~\ref{lem_not:units}, hence $p\in S$ where $p$ is the rational prime lying under $\p$. Together with Lemma~\ref{lem_not:units}
this implies the first statement of Theorem~\ref{th:mult_ind},
\textit{i.e.} $\emptyset\subsetneq S\subset S_a$. Let us fix the prime ideal
$\p$ and the rational prime $p$ lying under $\p$. Computing
norms we deduce that some positive power of $p$ divides $\Phi_k(m)$ as
well as $\Phi_k(m+a)$.

Next, we aim to prove:

\begin{proposition}\label{prop:not_divisible}
 For an odd prime $p\in S$ we have $p \nmid k$. In case that $p=2 \in S$ we have either $4\nmid k$ or $k=4$, $m=-1$ and $a=2$.
\end{proposition}

Before we state the proof we need the following useful
\begin{lemma}\label{lem:cyc_pol_div}
 Let $p$ be an odd prime with $p|k$ or $p=2$ and $4|k$, then
 $$\ell=\nu_p(\Phi_k(m))\leq 1$$
 for any integer $m$.
\end{lemma}

\begin{proof}
  The case that $\gcd(m,p)>1$ is trivial, since in this case $p|m$ and
  $\Phi_k(m)\equiv 1 \mod p$. Furthermore the case $m=\pm 1$ may be
  excluded since it is a consequence of
  Lemma~\ref{lem:Cyclotomic_pol}.

  Therefore we may concentrate on the case where $\gcd(m,p)=1$ and
  $m\neq\pm1$. First, let us assume that $p$ is an odd prime. If
  $q^2|k$ for some prime $q$, then we know that
  $\Phi_k(m)=\Phi_{k/q}(m^q)$. Thus we may assume that $k$ is
  square-free. We write $\Phi_k(m)=p^\ell A$ for some integer $A$, with $p\nmid A$. Let us
  note that
  \[\Phi_k(x)=\prod_{d|k} (x^d-1)^{\mu(k/d)}\]
  and recall that $p|m^d-1$ if and only if $\ord_p(m)|d$. If we
  assume for the moment that $\ord_p(m)\nmid k$, then $m^d-1$ and $p$
  are coprime for all divisors $d|k$ and therefore $\ell=0$. Now we
  assume that $\ord_p(m)|k$. Since $\gcd(p,\ord_p(m))=1$ there also
  exists a proper divisor $d|k$ such that $\ord_p(m)|d$. We obtain
\begin{equation}\label{eq:Phi_k_p-adic}
\prod_{\ord_p(m)|d|k} (m^d-1)^{\mu(k/d)} = p^\ell \underbrace{A \prod_{\ord_p(m)\nmid d|k} (m^d-1)^{-\mu(k/d)}}_{=:A'}.
\end{equation}
Note that $\nu_p(A')=0$, \textit{i.e.} the $p$-adic valuation of the right side of \eqref{eq:Phi_k_p-adic} is $\ell$. Let us compute the $p$-adic 
valuations on the left side of \eqref{eq:Phi_k_p-adic}.
Therefore we use the formula
\begin{equation}\label{eq:p-adic-formula}
 \nu_p(m^d-1)=m_p+\nu_p(d),
\end{equation}
where for odd primes $p$ we have $m_p=\nu_p(m^{\ord_p(m)}-1)$. For a proof of formula \eqref{eq:p-adic-formula} see e.g. \cite[Section 2.1.4, Lemma 
2.1.22]{Cohen:NTI}. Since we assume that $k$ is square-free $\nu_p(m^d-1)=m_p+1$ or $m_p$, depending 
whether $p|d$ or not. In particular the $p$-adic valuation on the left side is
$$\underbrace{\sum_{\ord_p(m)|d|k} m_p\mu(k/d)}_{=0}+\sum_{p\cdot \ord_p(m)|d|k} \mu(k/d)=\left\{ \begin{array}{cl} 1 & \quad \text{if}\;\; 
p\cdot \ord_p(m)=k\\
0 & \quad \text{else}\end{array} \right.     
$$
which proves the lemma for odd primes $p$.

Now, we turn to the case $p=2$. Let us assume first that $k$ is not a power of $2$. As above we may assume that $k$ is square-free. Moreover, since 
$\Phi_{2k}(x)=\Phi_k(-x)$ for an odd integer $k>1$ we may assume that $k>1$ is odd and square-free (note that we assume that $k$ is not a power of $2$). Let us 
note that $\nu_2(m^d-1)=\nu_2(m-1)$, provided $d$ is odd (see e.g.\cite[Section 2.1.4, Corollary 2.1.23]{Cohen:NTI}).
Similarly as in the case $p>2$ we consider 
$$\prod_{d|k} (m^d-1)^{\mu(k/d)} = 2^\ell A,\quad \text{with $A$ odd}$$
and aim to compute $\ell$.
Since we assume that $k$ is odd also every divisor $d$ of $k$ is odd and we obtain
$$\ell=\nu_2\left(\prod_{d|k} (m^d-1)^{\mu(k/d)}\right)= \nu_2(m-1) \sum_{d|k}\mu(k/d)=0.$$

We are left with the case that $k=2^n$ for some $n\geq 2$ and $m$ is odd. By \cite[Section 2.1.4, Corollary 
2.1.23]{Cohen:NTI} we have for even integers $d>0$ that $\nu_2(m^d-1)=m_2+\nu_2(d)$, with 
$$m_2=\left\{ \begin{array}{cl} \nu_2(m-1) &\quad \text{if}\;\; m\equiv 1 \mod 4\\
               \nu_2(m^2-1)-1 &\quad \text{else}
              \end{array}\right. .$$
Thus we obtain 
$$\ell=\nu_2(\Phi_{2^n}(m))=\nu_2\left(\frac{m^{2^n}-1}{m^{2^{n-1}}-1}\right)=m_2+n-(m_2+n-1)=1.$$
\end{proof}

Now we return to the
\begin{proof}[Proof of Proposition \ref{prop:not_divisible}]
  On the contrary let us assume that $p$ is an odd prime such that
  $p|k$ or $p=2$ and $4|k$. Then Lemma \ref{lem:cyc_pol_div} implies
  that $\Phi_k(m)=\pm\Phi_k(m+a)$. But this also yields
\begin{equation}\label{eq:almost_equal}
-m+\zeta_k=(-(m+a)+\zeta_k)\zeta_k^\ell
\end{equation}
for some integer $\ell$.  To show that this is impossible let us note
that
$$\left|\frac{\Phi_k(m+a)}{\Phi_k(m)}\right|=\prod_{1\leq n\leq k \atop \gcd(n,k)=1}\left|\frac{m+a-\zeta_k^n}{m-\zeta_k^n}\right|$$
It is easy to see that
$\left|\frac{m+a-\zeta_k^n}{m-\zeta_k^n}\right|>1$ for all integers
$m>\frac{-a+1}2$ and
$\left|\frac{m+a-\zeta_k^n}{m-\zeta_k^n}\right|<1$ for all integers
$m<\frac{-a-1}2$. Therefore we are left with the following two cases:
\begin{itemize}
 \item $a$ is even and $m=-a/2$, or
 \item $a$ is odd and $m=\frac{-a\pm 1}2$.
\end{itemize}
 
First, let us assume that $a$ is even and $p\in S$ is odd. Since $p|a$ and also $p|\frac
a2$ we have
$$\gcd(p,\Phi_k(-a/2))=1.$$
The same argument applies if $p=2$ and $4|a$. Let us assume that $p=2$ and $a\equiv 2\mod 4$. In this case we have
$$\left|\frac{m+a-\zeta_k}{m-\zeta_k}\right|=\left|\frac{a-2\zeta_k}{a+2\zeta_k}\right|\neq 1$$ provided $\Re(\zeta_k)\neq 0$.
But $\Re(\zeta_k)= 0$ implies $k=4$ and by \eqref{eq:almost_equal} we obtain 
\[(m+i)=(-m+i) i^\ell,\]
which yields $|m|\leq 1$. And indeed $1+i$ and $-1+i$ are multiplicatively dependent, which yields the exception in Proposition \ref{prop:not_divisible} 
and Theorem \ref{th:mult_ind} respectively.

Let us consider the case that $a$ is odd and $m=\frac{-a+1}2$. Note that in
this case 
$$\left|\frac{m+a-\zeta_k}{m-\zeta_k}\right|=\left|\frac{a+1-2\zeta_k}{-a+1-2\zeta_k}\right| >1$$
if the real
part of $\zeta_k$ is larger than $1/2$ and
$$\left|\frac{m+a-\zeta_k}{m-\zeta_k}\right|=\left|\frac{a+1-2\zeta_k}{-a+1-2\zeta_k}\right|<1$$
if the real part of
$\zeta_k$ is less than $1/2$. Therefore $\frac{a-1}2+\zeta_k$ and
$\frac{-a+1}2-\zeta_k$ are multiplicatively independent unless
$\Re(\zeta_k)= 1/2$. Hence we are left with the case $k=6$ and due to \eqref{eq:almost_equal} we have to
solve the Diophantine equation
$$\frac{a-1}2+\zeta_6=\left(\frac{-a-1}2+\zeta_6\right)\zeta_6^\ell.$$
Solving this equation for $\ell=0,\dots,5$ we get either a
contradiction or $a=\pm 1, \pm 3$. However in all cases $-m+\zeta_6$ is
a root of unity.  The case that $a$ is odd and $m=\frac{-a-1}2$ runs
along the same arguments and therefore we omit it.
\end{proof}

Let us summarize our results so far. Unless $a=2, m=-1$ and $k=4$ we have:
\begin{itemize}
 \item $\emptyset\subsetneq S \subset S_a$
 \item for all $p\in S$ we have $p\nmid k$.
\end{itemize}

The following proposition will complete the proof of Theorem \ref{th:mult_ind}:

\begin{proposition}\label{prop:res_class_deg}
Given the rational prime $p\in S$ and let $n=\nu_p\left( \gcd(\Phi_k(m),\Phi_k(m+a))\right)$. Then we have $n\leq \left\lfloor \frac{\nu_p(a)}{f_p} 
\right\rfloor f_p$ and $f_p|n$, where $f_p=\ord_k(p)$.
\end{proposition}

For consistency we continue with the proof of Theorem
\ref{th:mult_ind} and postpone the proof of
Proposition~\ref{prop:res_class_deg} to the end of this section.


Note that due to Proposition \ref{prop:not_divisible} and our assumption that $k$ is odd or $4|k$ we know that $p\nmid k$ and therefore $f_p$ is well-defined.

If $p^{\nu_p(a)} < k$ we obviously have $f_p=\ord_{k}(p)>\nu_p(a)$. Therefore Proposition \ref{prop:res_class_deg} yields that
$0=n\leq \left\lfloor \frac{\nu_p(a)}{f_p} \right\rfloor f_p$, \textit{i.e.} $p\not \in S$ if we have 
$p^{\nu_p(a)} < k$. Thus the second statement of Theorem~\ref{th:mult_ind} is proved, \textit{i.e.} $k\leq \min_{p\in S}\{p^\alpha\: :\: p^\alpha \| a\}$.
The third statement of Theorem~\ref{th:mult_ind} is now an immediate consequence of
Proposition~\ref{prop:res_class_deg}. Note that $\Phi_k(m)|\Phi_k(m+a)$ or $\Phi_k(m+a)|\Phi_k(m)$ if $-m+\zeta_k$ and $-(m+a)+\zeta_k$ are mulitplicatively 
dependent. Thus we may assume that $\Phi_k(x)=\gcd(\Phi_k(m),\Phi_k(m+a))$, with $x=m$ or $x=m+a$.

\begin{proof}[Proof of Proposition \ref{prop:res_class_deg}]
  Let us consider the ideal $I=(-m+\zeta_k,-m-a+\zeta_k)$ in the
  maximal order $\Z[\zeta_k]$. We note that
  $N(I)=\gcd(\Phi_k(m),\Phi_k(m+a))$, where $N(I)$ denotes the norm
of the ideal $I$. Moreover we know that
  $I=(a,-m+\zeta_k)$ and therefore the $\Z$-module $I'=\langle a,
  -m+\zeta_k, (-m+\zeta_k)^2, \dots, (-m+\zeta_k)^{\phi(k)-1}\rangle$ is
  a $\Z$-submodule of $I$. Of course $I'$ is also a sublattice in
  $\Z[\zeta_k]$ and has index
$$[\Z[\zeta_k]:I']=\det\left(\begin{array}{cccc} a & -m & \dots & *\\ 0 & 1 & \dots & * \\ \vdots & & \ddots & \vdots \\ 0& 0 & \dots & 1\end{array}\right)=a.$$
Since $\Z[\zeta_k]\supset I \supset I'$ we deduce from Lagrange's
theorem that $N(I)=[\Z[\zeta_k]:I]|a$. Let us fix a rational prime $p\in S$ and let
$\mathfrak p|(p)$ denote a prime ideal lying above the rational prime
$p$. Then we have
$$ \sum_{\mathfrak p|(p)} \nu_{\mathfrak p}(I)f_{\mathfrak p}=\nu_p(N(I))=n\leq \nu_p(a)$$
where $f_{\mathfrak p}$ is the residue class degree of $\mathfrak p$
in $\Q(\zeta_k)/\Q$. Let us note that it suffices to show that
$f_{\mathfrak p}=f_p$ for all prime ideals $\mathfrak p|(p)$. Since
$\Q(\zeta_k)/\Q$ is a Galois extension all $f_{\mathfrak p}$ are
equal. Therefore it is enough to show that the residue class degree $f_\p$ of a
particular $\mathfrak p$ in $\Q(\zeta_k)/\Q$ is exactly $\ord_k(p)=f_p$. However, this is well-known
from algebraic number theory and we summarize the crucial facts in the following lemma:

\begin{lemma}\label{lem:alg_nt}
 Let $p$ be a rational prime and let $k>2$ be an integer such that $p\nmid k$. Then
 \begin{itemize}
  \item $(p)=\p^{\phi(p^\ell)}$ is totally ramified in $\Z[\zeta_{p^\ell}]$ and $\p$ has therefore residue class degree $f_{\p}=1$.
  \item $(p)=\p_1 \dots \p_g$ is the prime idal factorisation in $\Z[\zeta_k]$, where $g=\phi(k)/f_\p$ and $f_\p=\ord_{k}(p)$ is the residue class degree of 
any of the prime ideals $\p_i$ for $i=1,\dots,g$.
 \end{itemize}
\end{lemma}

 The lemma is proved in most of the algebraic number theory books available, e.g. \cite[Chapter~11.3, Lemma N]{Ribenboim:ANT} for the first statement and
 \cite[Theorem 2.13]{Washington:CF} for the second.\end{proof}

\section{Proof of Corollary \ref{Cor:small_a}}\label{Sec:Algorithm}

Given an integer $a$, an application of Theorem \ref{th:mult_ind}
enables us (at least in principal) to find all pairs $(m,k)$ such that
$-m+\zeta_k$ and $-m-a+\zeta_k$ are multiplicatively
independent. Finding all pairs $(m,k)$ essentially depends on
computing a list of possible $k$'s. Obviously an integer $k$ satisfying the conditions of the second point of Theorem \ref{th:mult_ind} satisfies
$$k\left|\prod_{q^\beta < M\atop q \in \P\setminus S}q^\beta,\right.$$
where $M=\min_{p\in S} \{p^\alpha\: :\:p^\alpha \| a\}$. But, if $a$
contains only large prime powers the list of possible $k$'s might get
rather long. To overcome this problem we fix a set $\emptyset
\subsetneq S\subset S_a$ and compute for this specific subset $S$ a
list of possible candidates for $k$. If $a$ is not too large the
number of subsets $S$ usually stays small. In case that $a\leq 10^6$
we know that $|S_a|\leq 7$. Therefore we may assume that $S$ is
fixed. In view of Proposition \ref{prop:res_class_deg} we know that
for a prime $q<M$ we have $\nu_q(k)\geq \beta$ only if
$\ord_{q^\beta}(p)\leq\ord_k(p)\leq \nu_p(a)$ for all $p\in S$. In
particular the list of possible $k$ is contained in the list of
divisors of \[K=\prod_{q<M \atop q\in \P\setminus S} q^{\beta_q}\]
where $\beta_q$ is the maximal exponent $\beta$ such that
$\ord_{q^\beta}(p)\leq \nu_p(a)$ for all $p\in S$. The quantity $K$
can be computed rather quickly and we may consider only those $k$ such
that $k|K$ and $\ord_k(p)\leq \nu_p(a)$ for all $p\in S$.  This
approach to compute a list of candidates for $k$ is rather fast and
reduces the number of possible $k$ considerably.

In case of $a$ being cube-free there is even a stronger criteria:

\begin{lemma}\label{lem:cubefree}
    Assume that $\nu_p(a)\leq 2$ for all $p\in S$. Then $k|G$, where
    $G$ is the $\gcd$ of all $p^{\nu_p(a)}-1$ with $p\in S$.
\end{lemma}

\begin{proof}
By Proposition \ref{prop:res_class_deg} and the discussion above we have $\ord_{q^\beta}(p)\leq \ord_k(p) \leq 2$
that is $q^\beta|(p-1)$ or $q^\beta|(p+1)$. More precisely for all $p\in S$ with $\nu_p(a)=1$ we have $q^\beta|(p-1)$ and for all $p\in S$ with $\nu_p(a)=2$ 
we have $q^\beta|(p-1)(p+1)=p^2-1$. So in any case each $q^\beta|k$ also divides $p^{\nu_p(a)}-1$ with $p\in S$, \textit{i.e.} $q^\beta|G$.
\end{proof}

Let us summarize this by an example. We choose $a=3^3*19*127$ and $S=\{3,19\}$ then we obtain $M=19$ and $K=2$, \textit{i.e.} $k|2$ which we excluded. In  
case of $S=\{3\}$ we obtain $M=27$ and $K=104$ and we have to deal with $5$ candidates for $k|K$, namely $4,8,13,52$ and $104$. Note that $26|104$ but $26\equiv 
2 \mod 4$, which we may exclude due to Remark \ref{rem:k_not_2}. But,
only $k=4,8,13$ satisfy $\ord_k(3)\leq 3=\nu_3(a)$. In case of
$S=\{19,127\}$ we may apply Lemma \ref{lem:cubefree} to get 
$k|\gcd(18,126)=18$, hence $k=3$ or $k=9$.

Therefore we may assume that both $S$ and $k$ are fixed. By the third
statement of Theorem~\ref{th:mult_ind} we have only finitely many
possible values $Y$ for $\Phi_k(x)=Y$ with $x=m$ or $x=m+a$. Solving
the polynomial equation $\Phi_k(x)=Y$ for given $k$ and $Y$ we find
all possible values for $m$. However it is far from trivial to find
all integral solutions to the equation $\Phi_k(x)=Y$ in case that $k$
is large. Therefore let us describe our approach to this problem.

As already noted at several other places  $q^2|k$ implies that $\Phi_k(x)=\Phi_{k/q}(x^q)$ and we also know that $\Phi_{2k}(x)=\Phi_k(x)$ for odd integers $k$. 
Therefore we may consider the polynomial equation $\Phi_{k'}(x)=Y$ where $k'$ is the odd-square-free part of $k$. Since 
$\Phi_{k'}(0)-Y=1-Y$ the possible integral solutions to $\Phi_{k'}(x)=Y$ are $x=\pm d$ with $d|Y-1$. And it seems easy to check all possible $d$'s whether they
yield a solution to the equation $\Phi_k(x)=Y$ or not. And indeed it is easy as long as the degree of $\Phi_{k'}$ stays considerable small, say 
$\deg(\Phi_{k'})=\phi(k')\leq 20000$. In case that 
$\phi(k')\geq 20000$ computing $\Phi_{k'}(d)$ takes some time and by numerical evidence considering the equation $\Phi_{k'}(x)=Y$ modulo several small primes 
to exclude a possible solution seems to be more suitable. In particular in case that $\phi(k')\geq 20000$ we check  $\Phi_{k'}(d)\equiv Y \mod P$ for all 
possible solutions $d$ and all primes $P<10^3$ in order to find a possible solution. In the unlikely case that a possible solution could not be excluded by 
this approach we verify it directly. Let us note that we chose the treshold $20000$ since it seems that for $k'$ with $\phi(k')\geq 20000$ the modulo $P$ 
approach is faster in our implementation.

For several instances it happened that we found a solution to $\Phi_k(x)=Y$ and therefore found a possible pair $(k,m)$ such that $-m+\zeta_k$ and 
$-(m\pm a)+\zeta_k$ are multiplicatively dependent. But if $-m+\zeta_k$ and $-(m\pm a)+\zeta_k$ are multiplicatively dependent, then so are $\Phi_k(m)$ and 
$\Phi_k(m\pm a)$. Thus there exist integers $r,s$ such that $\Phi_k(m)^r=\Phi_k(m\pm a)^s$. But for a given pair of rational integers it is easy to check 
whether they are multiplicatively independent or not. 

An implementation of all these ideas leads to a rather efficient algorithm (see the Appendix for a concrete implementation). Let us note that the running time 
varies dramatically with $a$. For instance the longest runtime for a single instance was $670$ seconds obtained by $a=942479$ while the running time for the 
instance $a=950462$ took less than $0.01$ seconds. The computation was performed on a standard computer by the computer algebra system SAGE \cite{sage} 
and was split up onto several kernels. The cumulative CPU-computing time was about $36$ days and $13$ hours.

\section*{Acknowledgment}

The authors want to thank Julien Bernat (Universit\'e de Lorraine) for
posing the question of fixing $k$ and varying $m$ and $n$, which
led to Theorem \ref{th:mult_ind}.





\end{document}